\documentclass[leqno]{amsart}
\usepackage[T1,T2A]{fontenc}
\usepackage[utf8]{inputenc}
\usepackage[english]{babel}
\usepackage{amsthm}
\usepackage{color}
\usepackage{amssymb}
\usepackage{amsmath,amsfonts,enumerate}
\usepackage{amsthm}
\usepackage{amsmath}
\usepackage{graphicx}
\usepackage{hyperref}

\numberwithin{equation}{section}

\newtheorem{thm}{Theorem}
\newtheorem{lem}[thm]{Lemma}

\newtheorem{definition}[thm]{Definition}
\newtheorem{corollary}[thm]{Corollary}

\newtheorem{proposition}[thm]{Proposition}

\begin{document}

\title[The boundedness of the Hilbert transformation]
{The boundedness of the Hilbert transformation from one rearrangement invariant
Banach space into another and applications}
\author{F. Sukochev}
\address{School of Mathematics and Statistics, University of New South Wales, Kensington, NSW, 2052, Australia}
\email{f.sukochev@unsw.edu.au}
\author{K. Tulenov}
\address{
Al-Farabi Kazakh National University, 050040 Almaty, Kazakhstan;
Institute of Mathematics and Mathematical Modeling, 050010 Almaty, Kazakhstan.}
\email{tulenov@math.kz}
\author{D. Zanin}
\address{School of Mathematics and Statistics, University of New South Wales, Kensington, NSW, 2052, Australia}
\email{d.zanin@unsw.edu.au}


\date{}
\begin{abstract} In this paper, we study the boundedness of the Hilbert transformation in Lorentz function spaces, thereby complementing classical results of Boyd. We also characterize the optimal range of a triangular truncation operator in Schatten-Lorentz ideals. These results further entail sharp commutator estimates and applications to operator Lipschitz functions in Schatten-Lorentz ideals.
\end{abstract}

\maketitle

\section{Introduction}
In his thesis \cite[Chapter II]{Bo2} (see also \cite{B}) Boyd posed a problem of finding necessary and sufficient conditions for the Hilbert transformation,
$$(\mathcal{H}x)(t)=p.v.\frac{1}{\pi}\int_{\mathbb{R}}\frac{x(s)}{t-s}ds,$$
to be bounded from one Banach rearrangement invariant space $E$ into another $F.$ He found that the embedding $E\subseteq F$ is a necessary condition \cite[Corollary 3.2.2]{Bo2}, and some sufficient conditions \cite[Theorems 3.5 and 4.4]{Bo2}. Boyd was able to resolve the problem in the special case $E=F$ \cite[Theorems 3.4 and 4.3]{Bo2} (see also \cite[Theorem 3.7]{B}, \cite[Theorem III.5.18, p. 154]{BSh}, and \cite[Theorem II.7.2, p. 156]{KPS}). However, for an arbitrary couple $(E,F)$ of Banach rearrangement invariant spaces the problem remains unsolved. In this paper, we present its full resolution in the special case when $E=\Lambda_{\phi},$ an arbitrary Lorentz space. In fact, our first main result, Theorem \ref{main th} (and Corollary \ref{opt. range th H}) below, identifies the optimal range for the Hilbert transformation on $\Lambda_{\phi}$ as another Lorentz space $\Lambda_{\psi}.$ Thus, $\mathcal{H}$ acts boundedly from $E=\Lambda_{\phi}$ into rearrangement invariant space $F$ if and only if $\Lambda_{\psi}\subseteq F,$ where $\psi$ is given by formula \eqref{psi function} below. Of course, if the function $\psi$ is equivalent to $\phi,$ then our result simply recovers that of Boyd by ensuring that $E=\Lambda_{\phi}$ has non-trivial Boyd indices.

Recall also that Boyd was also able to characterise the boundedness of the Hilbert transformation for the spaces $\Lambda_{\phi,p}$, where $1<p<\infty$, the $p$-convexification of $\Lambda_{\phi}$ (see \cite[Theorem 2.4, p. 121]{Bo2} and \cite[Theorem 4.2]{B}). Namely, he proved that $\mathcal{H}$ acts boundedly on $\Lambda_{\phi,p},$ $1<p<\infty$ if and only if the latter space is uniformly convex. Thus, Theorem \ref{main th} below may be viewed as a complement to this Boyd's result in the crucial remaining case when $p=1.$

The second part of the paper is concerned with a non-commutative analogy of Boyd's question, when the couple of Banach rearrangement invariant spaces $(E,F)$ is replaced with $(\mathcal{E}(H),\mathcal{F}(H))$, the Banach ideals in $B(H)$, and the singular Hilbert transformation $\mathcal{H}$ is replaced by the triangular truncation operator (see \cite{GK1,GK2}) on integral operators (on the Hilbert space $L_2(\mathbb{R})$), defined by the formula
\begin{equation}\label{T-oper1}
(T(V)x)(t)=\int_{\mathbb{R}}\mathcal{K}(t,s){\rm sgn}(t-s)x(s)ds,\quad x\in L_2(\mathbb{R}),
\end{equation}
where $\mathcal{K}$ is the kernel of the integral operator $V,$
$$(Vx)(t)=\int_{\mathbb{R}}\mathcal{K}(t,s)x(s)ds,\quad x\in L_2(\mathbb{R}).$$
In the setting when $\mathcal{E}(H)=\mathcal{F}(H),$ the full analogy of Boyd's fundamental result was obtained by J. Arazy \cite[Theorem 4.1]{Ar} (see also \cite{GK1,GK2} for the operator $T$). However, as in the commutative case, when $\mathcal{E}(H)\neq\mathcal{F}(H)$ the problem remains unsolved. In our recent work \cite{STZ} this problem was treated in the somewhat different setting of quasi-Banach symmetrically normed ideals. We also refer to \cite{STZ} for detailed discussion of the deep interconnections of this problem with the double operator integration theory and its applications.

Our second main result, Theorem \ref{opt. range th T}, shows that the operator $T$ maps the Lorentz ideal $\Lambda_{\phi}(H)$ into the ideal $\mathcal{F}(H)$ if and only if $\Lambda_{\psi}(H)\subset \mathcal{F}(H),$ where $\psi$ is given by formula \eqref{psi function}. The applications of this result to Lipschitz and commutator estimates are given in Section 5.

\section{Preliminaries}

Throughout this paper, we write $\mathcal{A}\lesssim \mathcal{B}$ if there is a constant $c_{abs}> 0$ such that
$\mathcal{A}\leq c_{abs}\mathcal{B}.$ We write $\mathcal{A}\approx \mathcal{B}$ if both $\mathcal{A}\lesssim \mathcal{B}$ and $\mathcal{A}\gtrsim \mathcal{B}$ hold, possibly with different
constants.

\subsection{Symmetric (quasi-)Banach function spaces and operator ideals}
Let $(I,m)$ denote the measure space $I = \mathbb{R}_{+},\mathbb{R}$ (resp. $I=\mathbb{Z}_+,\mathbb{Z}$), where $\mathbb{R}_{+}:=(0,\infty)$ (resp. $\mathbb{Z}_+:=\mathbb{Z}_{\geq 0}$) and $\mathbb{R}$ is the set of real (resp. $\mathbb{Z}$ the set of integer) numbers, equipped with Lebesgue measure (resp. counting measure) $m.$
 Let $L(I,m)$ be the space of all measurable real-valued functions (resp. sequences) on $I$ equipped with Lebesgue measure (resp. counting measure) $m$ i.e. functions which coincide almost everywhere are considered identical. Define $L_0(I)$ to be the subset of $L(I,m)$ which consists of all functions (resp. sequences) $x$ such that $m(\{t : |x(t)| > s\})$ is finite for some $s > 0.$

For $x\in L_0(I)$ (where $I = \mathbb{R}_{+}$ or $\mathbb{R}$), we denote by $\mu(x)$ the decreasing rearrangement of the function $|x|.$ That is,
$$\mu(t,x)=\inf\{s\geq0:\ m(\{|x|>s\})\leq t\},\quad t>0.$$
If $I = \mathbb{Z}_+, \mathbb{Z}$, and $m$ is the counting measure, then $L_0(I) = \ell_\infty (I)$, where $\ell_\infty (I)$ denotes the space of all bounded sequences on $I$. In this case, for a sequence $x=\{x_n\}_{n\geq0}$ in $\ell_{\infty}(\mathbb{Z}_{+})$ (resp. $\ell_{\infty}(\mathbb{Z})$), we denote by $\mu(x)$ the sequence $|x|=\{|x_{n}|\}_{n\geq0}$ rearranged to be in decreasing order.

\begin{definition}\label{Sym} We say that $(E,\|\cdot\|_E)$ is a symmetric (quasi-)Banach function  space on $I$ if the following holds:
\begin{enumerate}[{\rm (a)}]
\item $E(I)$ is a subset of $L_0(I);$
\item $(E(I),\|\cdot\|_{E(I)})$ is a (quasi-)Banach space;
\item If $x\in E(I)$ and if $y\in L_0(I)$  are such that $|y|\leq|x|,$ then $y\in E(I)$ and $\|y\|_{E(I)}\leq\|x\|_{E(I)};$
\item If $x\in E(I)$ and if $y\in L_0(I)$ are such that $\mu(y)=\mu(x),$ then $y\in E(I)$ and $\|y\|_{E(I)}=\|x\|_{E(I)}.$
\end{enumerate}
\end{definition}
It is well known that $L_{p}(I)$ (resp. $\ell_p(I)$), $(0<p \leq\infty)$ is a basic example of (quasi-)Banach symmetric spaces of functions (resp. sequences).

Let $E$ be a symmetric Banach function space on $I.$
For $s>0,$ the dilation operator $D_{s}: E\to E$ is defined
by setting
$$D_{s}f(t)=f(t/s), \quad t>0, f\in E.$$
The upper and lower Boyd indices of $E(I)$ are numbers $\overline{\beta}_{E}$ and $\underline{\beta}_{E}$ defined by
$$
\overline{\beta}_{E}:=\lim_{s\downarrow 0+}\frac{\log\|D_{s}\|_{E\rightarrow E}}{\log s}, \,\ \underline{\beta}_{E}:=\lim_{s\rightarrow \infty}\frac{\log\|D_{s}\|_{E\rightarrow E}}{\log s}.
$$
Moreover, they satisfy $0\leq\underline{\beta}_{E}\leq\overline{\beta}_{E}\leq 1$
(see \cite[Definition III.5.12 and Proposition III.5.13, p. 149]{BSh}).
As is easily checked, if $E=L_p,$ $1\leq p\leq\infty,$
then $\underline{\beta}_{E}=\overline{\beta}_{E}=1/p.$ It is by now a classical result of D. W. Boyd \cite{Bo2, Bo1} that if $1\leq p<q\leq\infty,$ then $E$ is an interpolation space for the pair
$(L_p, L_q)$ if and only if $1/q\leq \underline{\beta}_{E}\leq\overline{\beta}_{E}<1/p.$ If $0<\underline{\beta}_{E}\leq\overline{\beta}_{E}<1,$ we shall say simply that $E$ has non-trivial Boyd indices.


One of the central results in the theory of interpolation of symmetric Banach
function spaces is the Boyd interpolation theorem which finds its roots in the
seminal work of Calder\'{o}n \cite{Cal}. For a given symmetric function space $E$ on $[0,\infty)$,
based on weak type estimates for a linear integral operator
introduced by Calder\'{o}n, Boyd \cite{Bo1} showed that every sublinear operator on $E$ which
is simultaneously of weak types $(p, p)$ and $(q, q)$ is bounded on $E$ if and only if
$1\leq p<\underline\beta_E\leq \overline\beta_E<q<\infty.$


Let $H$ denote a fixed separable Hilbert space and let $B(H)$ be the algebra of all bounded operators on $H.$ Let us denote by $K(H)$ the ideal of compact operators on $H$ and $\big\{\mu(n,A)\big\}_{n\in \mathbb{Z}_{+}}$ is the sequence of singular values of a compact operator $A$ (see \cite[Chapter II]{GK1}).

\begin{definition}\label{NC Sym} Let $\mathcal{E}(H)$  be a linear subset in $B(H)$ equipped with a complete norm $\|\cdot\|_{\mathcal{E}(H)}.$
We say that $\mathcal{E}(H)$ is a symmetric (quasi-)Banach operator ideal (in $B(H)$) if
for $A\in \mathcal{E}(H)$ and for every $B\in B(H)$ with $\mu(n,B)\leq\mu(n,A),$ $\forall n\in\mathbb{Z}_{+},$ we have $B\in \mathcal{E}(H)$ and $\|B\|_{\mathcal{E}(H)}\leq\|A\|_{\mathcal{E}(H)}$.
\end{definition}

Recall the construction of a symmetric (quasi-)Banach operator ideal (or non-commutative symmetric (quasi-)Banach ideal) $\mathcal{E}(H).$ Let $E$ be a symmetric Banach sequence space on $\mathbb{Z}_{+}.$ Set
$$\mathcal{E}(H)=\Big\{A\in K(H):\ \mu(A)\in E(\mathbb{Z}_{+})\Big\}.$$
We equip $\mathcal{E}(H)$ with a natural norm
$$\|A\|_{\mathcal{E}(H)}:=\|\mu(A)\|_{E(\mathbb{Z}_{+})},\quad A\in \mathcal{E}(H).$$
Set
$$\mathcal{E}(H)=\Big\{A\in K(H):\ \mu(A)\in E(\mathbb{Z}_{+})\Big\}.$$
It was proved in \cite{KS} (see also \cite[Question 2.5.5, p. 58]{LSZ}) that so defined $(\mathcal{E}(H),\|\cdot\|_{\mathcal{E}(H)})$
is a symmetric (quasi-)Banach operator ideal.
An extensive discussion of the various properties of such spaces can be found
in \cite{KS,LSZ}.
In particular, if $E=\ell_{p},$ $1\leq p<\infty,$ then we obtain
$$\mathcal{L}_{p}(H)=\Big\{A\in K(H):\ \mu(A)\in \ell_{p}(\mathbb{Z}_{+})\Big\},$$
which is so called Schatten-von Neumann class of all compact operators $A:H\rightarrow H$ with finite norm
$$\|A\|_{\mathcal{L}_{p}(H)}:=\big(\sum_{k=0}^{\infty}\mu(k,A)^{p}\big)^{1/p}.$$
If $p=\infty,$ then we set $\mathcal{L}_{\infty}(H):=B(H)$ with the uniform norm, i.e.  $\|A\|_{\mathcal{L}_{\infty}(H)}:=\|A\|,$ $A\in B(H).$
It is easy to see that these are symmetric Banach ideals (see \cite{LSZ} for more details).

\subsection{Lorentz spaces}\label{NC Lorentz}
\begin{definition}\label{quasiconcave definition}\cite[Definition II. 1.1, p. 49]{KPS} A function $\varphi:[0,\infty)\rightarrow [0,\infty)$ is said to be quasiconcave if
\begin{enumerate}[{\rm (i)}]
\item $\varphi(t)=0\Leftrightarrow t=0;$
\item $\varphi(t)$ is positive and increasing for $t>0;$
\item $\frac{\varphi(t)}{t}$ is decreasing for $t>0.$
\end{enumerate}
\end{definition}
Observe that every nonnegative concave function on $[0,\infty)$ that vanishes only at origin is quasiconcave. The reverse, however, is not always true. But, we may replace, if necessary, a quasiconcave function $\varphi$ by its least concave majorant $\widetilde{\varphi}$  such that
$$\frac{1}{2}\widetilde{\varphi}\leq\varphi\leq\widetilde{\varphi}$$
 (see \cite[Proposition 5.10, p. 71]{BSh}).

Let $\Omega$ denote the set of increasing concave functions $\varphi:[0,\infty)\rightarrow[0,\infty)$ for which $\lim_{t\rightarrow 0+}\varphi(t)=0$ (or simply $\varphi(0+)=0$).
For a function $\varphi$ in $\Omega,$ the Lorentz space $\Lambda_{\varphi}(I)$ is defined by setting
$$\Lambda_{\varphi}(I):=\left\{x\in L_{0}(I): \int_{\mathbb{R}_{+}}\mu(s,x)d\varphi(s)<\infty\right\}$$
equipped with the norm
\begin{equation}\label{Lorentz norm}
\|x\|_{\Lambda_\varphi(I)}:=\int_{\mathbb{R}_{+}}\mu(s,x)d\varphi(s).
\end{equation}
In particular, if $\varphi(t)=\log(1+t),\, t>0,$ then we denote $\Lambda_{\varphi}(I)$ by $\Lambda_{\log}(I).$

The Lorentz sequence space $\Lambda_{\varphi}(\mathbb{Z}_{+})$ is defined as follows
\begin{equation}\label{Lorentz sec}
\Lambda_{\varphi}(\mathbb{Z}_{+}):=\left\{a\in c_{0}:\|a\|_{\Lambda_{\varphi}(\mathbb{Z}_{+})}=\sum_{n=0}^{\infty}\mu(n,a)(\varphi(n+1)-\varphi(n))<\infty\right\},
\end{equation}
where $c_{0}$ is the space of sequences converging to zero.
Similarly, if $\varphi(t):=\log(1+t), \, t>0,$ then $\Lambda_{\varphi}(\mathbb{Z}_{+})=\Lambda_{\log}(\mathbb{Z}_{+}).$
For more details on Lorentz spaces, we refer the reader to \cite[Chapter II.5]{BSh} and \cite[Chapter II.5]{KPS}.

The Lorentz ideal $\Lambda_{\varphi}(H)$ (see \cite[Example 1.2.7, p. 25]{LSZ}) is defined as follows
$$\Lambda_{\varphi}(H):=\left\{A\in K(H):\|A\|_{\Lambda_{\varphi}(H)}=\sum_{n=0}^{\infty}\mu(n,A)(\varphi(n+1)-\varphi(n))<\infty\right\}.$$
Let $\varphi(t):=\log(1+t),$ $ t>0.$ Then the corresponding Lorentz ideal $\Lambda_{\log}(H)$ is defined by
$$\Lambda_{\log}(H):=\left\{A\in K(H):\|A\|_{\Lambda_{\log}(H)}=\sum_{n=0}^{\infty}\frac{\mu(n,A)}{n+1}<\infty\right\}.$$

%

\subsection{Calder\'{o}n operator and Hilbert transformation}\label{calderon}

 For a function $x\in \Lambda_{\log}(\mathbb{R}_{+}),$ define the Calder\'{o}n operator $S:\Lambda_{\log}(\mathbb{R}_{+})\to (L_{1,\infty}+L_{\infty})(\mathbb{R}_{+})$ as follows
\begin{equation}\label{S}
 (Sx)(t):=\frac1t\int_0^tx(s)ds+\int_t^{\infty}x(s)\frac{ds}{s}, \,\ t>0.
\end{equation}
It is obvious that $S$ is a linear operator.
If $0<t_{1}<t_{2},$ then

$$\min\left(1,\frac{s}{t_{2}}\right)\leq \min\left(1,\frac{s}{t_{1}}\right)\leq \frac{t_{2}}{t_{1}}\cdot\min\left(1,\frac{s}{t_{2}}\right), \,\, s>0.$$
Let $x$ be a nonnegative function on $[0,\infty).$ It follows from the first of preceding inequalities that $(Sx)(t)$ is a
decreasing function of $t.$
The operator $S$ is often applied to the decreasing rearrangement $\mu(x)$ of a function $x$ defined on some other measure space.
Since $S\mu(x)$ is non-increasing itself, it is easy to see that $\mu\big(S\mu(x)\big)=S\mu(x).$

\begin{proposition} \label{S equi} Let $S$ be the operator defined as in \eqref{S}. If
\begin{equation}\label{fy}\varphi_{0}(t):=\left\{
                   \begin{array}{ll}
                     t\log(\frac{e^{2}}{t}), & 0<t<1 ,\\
                     2\log(et), & 1\leq t<\infty ,
                   \end{array}
                 \right.
\end{equation} then the Lorentz space $\Lambda_{\varphi_{0}}(\mathbb{R}_{+})$ is the maximal among the symmetric Banach spaces $E(\mathbb{R}_{+})$ such that
$$S:E(\mathbb{R}_{+})\rightarrow(L_{1}+L_{\infty})(\mathbb{R}_{+}).$$
\end{proposition}
\begin{proof} Let $E(\mathbb{R}_{+})$ be a symmetric Banach function space such that $S:E(\mathbb{R}_{+})\rightarrow (L_{1}+L_{\infty})(\mathbb{R}_{+}).$ First, we claim that the following equivalence holds for every $x\in\Lambda_{\varphi_{0}}(\mathbb{R}_{+})$
\begin{equation}\label{equ S}
  \|S\mu(x)\|_{(L_{1}+L_{\infty})(\mathbb{R}_{+})}\leq \|x\|_{\Lambda_{\varphi_{0}}(\mathbb{R}_{+})}\leq 2 \|S\mu(x)\|_{(L_{1}+L_{\infty})(\mathbb{R}_{+})}.
\end{equation}
Indeed,
if $x\in\Lambda_{\varphi_{0}}(\mathbb{R}_{+}),$ then by using Fubini's theorem and \eqref{fy}, we obtain
\begin{eqnarray*}\begin{split}
\|S\mu(x)\|_{(L_{1}+L_{\infty})(\mathbb{R}_{+})}
&=\int_{\mathbb{R}_{+}}S\mu(t,x)dt\stackrel{\eqref{S}}=\int_{\mathbb{R}_{+}}\frac{1}{t}\int_{0}^{t}\mu(s,x)dsdt+\int_{\mathbb{R}_{+}}\int_{t}^{\infty}\frac{\mu(s,x)}{s}dsdt\\
&=\int_{0}^{1}\mu(s,x)\big(1-\log(s)\big)ds+\int_{1}^{\infty}\mu(s,x)\frac{1}{s}ds\\
&\leq\int_{0}^{1}\mu(s,x)\big(1-\log(s)\big)ds+2\int_{1}^{\infty}\mu(s,x)\frac{1}{s}ds\stackrel{\eqref{Lorentz norm}}=\|x\|_{\Lambda_{\varphi_{0}}(\mathbb{R}_{+})}.\\
\end{split}\end{eqnarray*}
On the other hand, we have
\begin{eqnarray*}\begin{split}
\|x\|_{\Lambda_{\varphi_{0}}(\mathbb{R}_{+})}
&\stackrel{\eqref{Lorentz norm}}=\int_{0}^{1}\mu(s,x)\big(1-\log(s)\big)ds+2\int_{1}^{\infty}\mu(s,x)\frac{1}{s}ds\\
&\leq 2\int_{0}^{1}\mu(s,x)\big(1-\log(s)\big)ds+2\int_{1}^{\infty}\mu(s,x)\frac{1}{s}ds\\
&\stackrel{\eqref{S}}=2\|S\mu(x)\|_{(L_{1}+L_{\infty})(\mathbb{R}_{+})}
\end{split}\end{eqnarray*}
as claimed.
Then, for any $x\in E(\mathbb{R}_{+}),$ it follows from the second inequality of \eqref{equ S} that
$$\|x\|_{\Lambda_{\varphi_{0}}(\mathbb{R}_{+})}\leq 2\|S\mu(x)\|_{(L_{1}+L_{\infty})(\mathbb{R}_{+})}\lesssim\|x\|_{E(\mathbb{R}_{+})}.$$
This shows $E(\mathbb{R}_{+})\subset\Lambda_{\varphi_{0}}(\mathbb{R}_{+}),$ thereby completing the proof.
\end{proof}

Since for each $t>0,$ the kernel
$k_{t}(s)=\frac{1}{s}\cdot\min\Big\{1,\frac{s}{t}\Big\}$
is a non-increasing function of $s,$ it follows  from \cite[Theorem II.2.2, p. 44]{BSh} that
\begin{equation}\begin{split}\label{S proper}
|(Sx)(t)|
&\stackrel{\eqref{S}}{=}\Big|\int_{\mathbb{R}_{+}}x(s)\min\Big\{1,\frac{s}{t}\Big\}\frac{ds}{s}\Big|\\
&\leq\int_{\mathbb{R}_{+}}|x(s)|\min\Big\{1,\frac{s}{t}\Big\}\frac{ds}{s}\\
&\leq\int_{\mathbb{R}_{+}}\mu(s,x)\min\Big\{1,\frac{s}{t}\Big\}\frac{ds}{s}\stackrel{\eqref{S}}{=}(S\mu(x))(t), \,\ \forall x\in \Lambda_{\log}(\mathbb{R}_{+}).
\end{split}\end{equation}

If $x\in \Lambda_{\log}(\mathbb{R}),$  then the classical Hilbert transformation $\mathcal{H}$ is defined by the principal-value integral
\begin{equation}\label{hilbert tr}
(\mathcal{H}x)(s):=p.v.\frac{1}{\pi}\int_{\mathbb{R}}\frac{x(\eta)}{s-\eta}d\eta, \,\ \forall x\in \Lambda_{\log}(\mathbb{R}),
\end{equation}
(see, e.g. \cite[Chapter III. 4]{BSh}).

The following lemma is taken from \cite{BSh} (see Proposition III. 4.10 on p. 140 there).

\begin{lem}\label{dom of Hilbert trans} Let $x=x\chi_{(0,\infty)}$ be such that $x$ is a non-negative decreasing function on $\mathbb{R}_{+}.$ We have
$$\mu(\mathcal{H}x)\geq\frac1{2\pi}S\mu(x).$$
If $\mathcal{H}x$ is everywhere defined, then $S\mu(x)$ exists, that is, $x$ belongs to the domain of $S,$ i.e. $x\in \Lambda_{\log}(\mathbb{R}_{+})$ (see \eqref{S}).
\end{lem}
\begin{proof} For every $t>0,$ we have
\begin{eqnarray*}\begin{split} |(\mathcal{H}x)(-t)|
&\stackrel{\eqref{hilbert tr}}{=}\frac{1}{\pi}\left|\int_{\mathbb{R}_{+}}\frac{x(s)}{-t-s}ds\right|\\
&=\frac{1}{\pi}\int_{\mathbb{R}_{+}}\frac{x(s)}{t+s}ds=\frac{1}{\pi}\left(\int_{0}^{t}\frac{x(s)}{t+s}ds+\int_{t}^{\infty}\frac{x(s)}{t+s}ds\right)\\
&\geq\frac{1}{\pi}\left(\int_{0}^{t}\frac{x(s)}{2t}ds+\int_{t}^{\infty}\frac{x(s)}{2s}ds\right)\\
&=\frac{1}{2\pi}\cdot\left(\frac{1}{t}\int_{0}^{t}x(s)ds+\int_{t}^{\infty}\frac{x(s)}{s}ds\right)\stackrel{\eqref{S}}{=}\frac{1}{2\pi}(Sx)(t).
\end{split}\end{eqnarray*}
\end{proof}

On the other hand, if $x\in\Lambda_{\log}(\mathbb{R}_{+}),$ then by \cite[Theorem III.4.8, p. 138]{BSh}, we have
$$\mu(\mathcal{H}x)\lesssim S\mu(x),$$
which shows existence of $\mathcal{H}x.$

Define the discrete version of the operator $S^{d}:\Lambda_{\log}(\mathbb{Z}_{+})\rightarrow \ell_{\infty}(\mathbb{Z}_{+})$ by
\begin{equation}\label{S dis}\big(S^{d}a\big)(n):=\frac{1}{n+1}\sum_{k=0}^{n}a(k)+\sum_{k=n+1}^{+\infty}\frac{a(k)}{k}, \,\,\ a\in \Lambda_{\log}(\mathbb{Z}_{+}).
\end{equation}

\subsection{Triangular truncation operator}
Our primary example is a triangular truncation operator on the Hilbert space $H=L_2(\mathbb{R}).$ More precisely, let $\mathcal{K}$ be a fixed measurable function on $\mathbb{R}\times\mathbb{R}.$ Let us consider an operator $V$ with the
integral kernel $\mathcal{K}$ on $L_2(\mathbb{R})$ defined by setting
\begin{equation}\label{int oper}(Vx)(t)=\int_{\mathbb{R}}\mathcal{K}(t,s)x(s)ds, \,\,\ x\in L_2(\mathbb{R}).
\end{equation}
We define the triangular truncation operator $T$ for any $V\in \mathcal{E}(H)$ as follows (see \cite{GK1,GK2} for more details).
\begin{equation}\label{T-oper}
(T(V)x)(t)=\int_{\mathbb{R}}\mathcal{K}(t,s){\rm sgn}(t-s)x(s)ds, \,\,\ x\in L_2(\mathbb{R}).
\end{equation}
 It was proved in \cite[Theorem 11]{STZ} that $T$ is bounded from $\mathcal{L}_{1}(H)$ into $\mathcal{L}_{1,\infty}(H).$

\section{Optimal range of the Calder\'{o}n operator and Hilbert transformation in Lorentz spaces}

In this section, we describe the optimal range for the Calder\'{o}n operator $S$ among the Lorentz spaces.
Consequently, we obtain similar result for the classical Hilbert transformation. Our main result in this section is Theorem \ref{main th} below, which identifies the optimal range of the Calder\'{o}n operator $S$ on $\Lambda_{\phi}$ as another Lorentz space $\Lambda_{\psi}$ together with Corollary \ref{main th2}.

Let $\phi:[0,\infty)\to[0,\infty)$ be an increasing concave function such that $\phi(0+)=0.$ Define a function $\psi$ by the following formula
\begin{equation}\label{psi function}\psi(u):=\inf_{w>1}\frac{\phi(uw)}{1+\log(w)}.
\end{equation}
We need the following lemmas.
\begin{lem}\label{psi def lemma} If a function $\phi:[0,\infty)\to[0,\infty)$ is increasing and concave such that $\phi(0+)=0,$ then the function $\psi$ defined by the formula
\eqref{psi function}
is also increasing and concave.
\end{lem}
\begin{proof} For each $w>1,$ the mapping
$$u\to\frac{\phi(uw)}{1+\log(w)}$$
is concave and increasing. Since infimum of an arbitrary family of concave (resp. increasing) functions is concave (resp. increasing), the assertion follows.
\end{proof}

\begin{lem}\label{L2} Let $\phi:[0,\infty)\to[0,\infty)$ be an increasing concave function such that $\phi(0+)=0$ and let $\psi$ be the function defined by the formula \eqref{psi function}.
Then
\begin{equation}\label{psi cond}
\lim_{u\to0}\log(\frac1u)\psi(u)=0.
\end{equation}
\end{lem}
\begin{proof} Let $\epsilon>0$ and let $u\in(0,\epsilon).$ Setting $w=\frac{\epsilon}{u}$ in the definition of $\psi,$ we obtain
$$\psi(u)\leq\frac{\phi(\epsilon)}{1+\log(\frac{\epsilon}{u})}.$$
Therefore,
$$\limsup_{u\to0}\log(\frac1u)\psi(u)\leq\limsup_{u\to0}\frac{\log(\frac1u)\phi(\epsilon)}{1+\log(\frac{\epsilon}{u})}=\phi(\epsilon).$$
Since $\epsilon$ can be chosen arbitrarily small, the assertion follows.
\end{proof}

The following result could be inferred from \cite[Lemma II.5.2 and Corollary II.1, pp. 111-112]{KPS} (see also \cite[Lemma 5]{SS} for finite measure space); however, we present its proof for convenience of the reader.
\begin{lem}\label{KPS lem} Let $\phi$ and $\varphi$ be increasing concave functions on $[0,\infty)$ vanishing at the origin.
 Suppose $\Lambda_{\phi}(\mathbb{R}_{+})\subset\Lambda_{\log}(\mathbb{R}_{+}).$
We have
$$S:\Lambda_{\phi}(\mathbb{R}_{+})\to\Lambda_{\varphi}(\mathbb{R}_{+})$$
if and only if
$$\|S\chi_{_{\Delta}}\|_{\Lambda_{\varphi}(\mathbb{R}_{+})}\leq c_{\phi,\varphi}\phi(m(\Delta)),\quad \Delta\subset\mathbb{R}_{+}.$$
\end{lem}
\begin{proof}  If $$S:\Lambda_{\phi}(\mathbb{R}_{+})\to\Lambda_{\varphi}(\mathbb{R}_{+}),$$ that is
$$\|Sf\|_{\Lambda_{\varphi}(\mathbb{R}_{+})}\leq c_{\phi,\varphi}\|f\|_{\Lambda_{\phi}(\mathbb{R}_{+})},\quad \forall f\in\Lambda_{\phi}(\mathbb{R}_{+}),$$
then taking $f=\chi_{\Delta},$ we obtain
$$\|S\chi_{_{\Delta}}\|_{\Lambda_{\varphi}(\mathbb{R}_{+})}\leq c_{\phi,\varphi} \|\chi_{_{\Delta}}\|_{\Lambda_{\phi}(\mathbb{R}_{+})}=c_{\phi,\varphi}\phi(m(\Delta)),\quad
\Delta\subset\mathbb{R}_{+}.$$

Conversely, let
\begin{equation}\label{assump}
\|S\chi_{_{\Delta}}\|_{\Lambda_{\varphi}(\mathbb{R}_{+})}\leq c_{\phi,\varphi}\phi(m(\Delta)),\quad \Delta\subset\mathbb{R}_{+}.
\end{equation}
 We split the proof into three steps.

{\bf Step 1.}
Let $x$ be a nonnegative simple function in $\Lambda_{\phi}(\mathbb{R}_{+}).$ Then it can be represented in the form
$$x(t)=\sum_{k=1}^{N}\alpha_{k}\chi_{_{\Delta_{k}}}(t),$$
where $\Delta_{1}\subset\Delta_{2}\subset...\subset\Delta_{N}$ and $\alpha_{k}>0,$ $k=1,2,...,N.$ Hence,
\begin{equation}\label{S simple}
\|Sx\|_{\Lambda_{\varphi}(\mathbb{R}_{+})}\leq\sum_{k=1}^{N}\alpha_{k}\|S\chi_{_{\Delta_{k}}}\|_{\Lambda_{\varphi}(\mathbb{R}_{+})}\leq c_{\phi,\varphi}\cdot\sum_{k=1}^{N}\alpha_{k}\phi(m(\Delta_{k}))
\end{equation}
On the other hand, by formula (2.3) in \cite[Chapter II, p. 60]{KPS}, we have
$$\mu(x)=\sum_{k=1}^{N}\alpha_{k}\chi_{[0,m(\Delta_{k})]},$$
and consequently
\begin{equation}\label{x simple}
\|x\|_{\Lambda_{\phi}(\mathbb{R}_{+})}=\int_{0}^{\infty}\sum_{k=1}^{N}\alpha_{k}\chi_{[0,m(\Delta_{k})]}(t)d\phi(t)=\sum_{k=1}^{N}\alpha_{k}\phi(m(\Delta_{k}))
\end{equation}
Combining \eqref{S simple} and \eqref{x simple}, we obtain
\begin{equation}\label{S positive}
\|Sx\|_{\Lambda_{\varphi}(\mathbb{R}_{+})}\leq c_{\phi,\varphi}\|x\|_{\Lambda_{\phi}(\mathbb{R}_{+})}.
\end{equation}

{\bf Step 2.}
If now the simple function $x$ from $\Lambda_{\phi}(\mathbb{R}_{+})$ has values of arbitrary sign, then by \eqref{S positive}
\begin{equation}\begin{split}\label{S arbitrary} \|Sx\|_{\Lambda_{\varphi}(\mathbb{R}_{+})}
&=\|S(x_{+}-x_{-})\|_{\Lambda_{\varphi}(\mathbb{R}_{+})}\\
&=\|Sx_{+}-Sx_{-}\|_{\Lambda_{\varphi}(\mathbb{R}_{+})}\\ &\leq\|Sx_{+}\|_{\Lambda_{\varphi}(\mathbb{R}_{+})}+\|Sx_{-}\|_{\Lambda_{\varphi}(\mathbb{R}_{+})}\\
&\leq c_{\phi,\varphi}(\|x_{+}\|_{\Lambda_{\phi}(\mathbb{R}_{+})}+\|x_{-}\|_{\Lambda_{\phi}(\mathbb{R}_{+})})\\
&\leq 2 c_{\phi,\varphi}\|x\|_{\Lambda_{\phi}(\mathbb{R}_{+})}.
\end{split}\end{equation}

{\bf Step 3.}
Let $f\in\Lambda_{\phi}(\mathbb{R}_{+})$ and let $f_{n}$ be a sequence of simple functions approximating $f.$ Then, by \eqref{S arbitrary} we obtain
$$\|S(f_{n}-f_{m})\|_{\Lambda_{\varphi}(\mathbb{R}_{+})}\leq c_{\phi,\varphi}\|f_n-f_m\|_{\Lambda_{\phi}(\mathbb{R}_{+})},$$
and so the sequence $Sf_{n}$ is Cauchy in $\Lambda_{\varphi}(\mathbb{R}_{+}).$ Note that $S$ is a positive operator, i.e. it maps nonnegative functions to nonnegative functions.
Since $S:\Lambda_{\log}(\mathbb{R}_{+})\to L_{0}(\mathbb{R}_{+})$ (see Subsection \ref{calderon}) and $\Lambda_{\phi}(\mathbb{R}_{+})\subset\Lambda_{\log}(\mathbb{R}_{+}),$ it follows from \cite[Proposition 1.3.5, p. 27]{MN} that $S$ is continuous from $\Lambda_{\phi}(\mathbb{R}_{+})$  into $L_{0}(\mathbb{R}_{+}).$ Therefore, the sequence $Sf_n$ converges to $Sf$ in $L_{0}(\mathbb{R}_{+}),$ and its limit will be the same $Sf$ in $\Lambda_{\varphi}(\mathbb{R}_{+}),$ i.e. $Sf\in\Lambda_{\varphi}(\mathbb{R}_{+})$ which shows
$$S:\Lambda_{\phi}(\mathbb{R}_{+})\to\Lambda_{\varphi}(\mathbb{R}_{+}).$$
\end{proof}

\begin{lem}\label{measurable set} Let $S$  be the operator defined in \eqref{S}. For any measurable set $\Delta\subset \mathbb{R}_{+}$
with measure $m(\Delta)=u>0,$ we have
$$S\chi_{\Delta}\leq 2S\chi_{(0,m(\Delta))}.$$
\end{lem}
\begin{proof}
Let $\Delta$ be a Lebesgue measurable set with measure $m(\Delta)=u>0.$ Then,
$$\big(S\chi_{\Delta}\big)(t)=\frac{1}{t}\int_{0}^{t}\chi_{\Delta}(s)ds+\int_{t}^{\infty}\frac{\chi_{\Delta}(s)}{s}ds.$$
We have
\begin{equation}\label{C oper}
\frac{1}{t}\int_{0}^{t}\chi_{\Delta}(s)ds\leq \frac{\min\{t,m(\Delta)\}}{t}=\frac{\min\{t,u\}}{t}=\min\{1,\frac{u}{t}\}, \,\ u,t>0.
\end{equation}
On the other hand, we have
$$
\int_{t}^{\infty}\frac{\chi_{\Delta}(s)}{s}ds\leq\int_{t}^{t+u}\frac{ds}{s}=\log\big(1+\frac{u}{t}\big), \,\ u,t>0.
$$
If $u>t,$ then
$$\log\left(1+\frac{u}{t}\right)\leq 1+\log\left(\frac{u}{t}\right),$$
and if $u\leq t,$ then
$$\log\left(1+\frac{u}{t}\right)\leq \frac{u}{t}, \,\ u,t>0.$$
Therefore, we have
\begin{equation}\label{C* oper}
\int_{t}^{\infty}\frac{\chi_{\Delta}(s)}{s}ds\leq\left\{
\begin{array}{ll}
\frac{u}{t} , & 0<u\leq t, \\
1+\log\big(\frac{u}{t}\big), &  t<u<\infty.
\end{array}
\right.
\end{equation}
Combining \eqref{C oper} and \eqref{C* oper}, we obtain
$$\big(S\chi_{\Delta}\big)(t)\leq \left\{
\begin{array}{ll}
\frac{2u}{t} , & 0<u\leq t, \\
2+\log\big(\frac{u}{t}\big), &  t<u<\infty.
\end{array}
\right.
$$
On the other hand, we know that
$$\big(S\chi_{(0,u)}\big)(t)=\left\{
\begin{array}{ll}
\frac{u}{t} , & 0<u\leq t, \\
1+\log\big(\frac{u}{t}\big), &  t<u<\infty.
\end{array}
\right.
$$
Thus, we have
$$\big(S\chi_{\Delta}\big)(t)\leq 2\big(S\chi_{(0,m(\Delta))}\big)(t), \,\ t>0.$$
\end{proof}

The following two lemmas are crucial.
\begin{lem}\label{L3}  Let $\phi$ and $\varphi$ be increasing concave functions on $[0,\infty)$ vanishing at the origin. Suppose that
$$\lim_{t\to0}\log\big(\frac{1}{t}\big)\varphi(t)=0,\quad \lim_{t\to\infty}\frac{\varphi(t)}{t}=0.$$
We have
$$S:\Lambda_{\phi}(\mathbb{R}_{+})\to\Lambda_{\varphi}(\mathbb{R}_{+})$$
if and only if
$$\int_0^u\frac{\varphi(t)}{t}dt+u\int_u^{\infty}\frac{\varphi(t)}{t^2}dt\leq c_{\phi,\varphi}\phi(u),\quad u>0.$$
Here $c_{\phi,\varphi}$ is a constant depending only on functions $\phi$ and $\varphi.$
\end{lem}
\begin{proof} By Lemma \ref{KPS lem}, we have
$$S:\Lambda_{\phi}(\mathbb{R}_{+})\to\Lambda_{\varphi}(\mathbb{R}_{+})$$ if and only if
$$\|S\chi_\Delta\|_{\Lambda_{\varphi}(\mathbb{R}_{+})}\leq c_{\phi,\varphi}\phi(m(\Delta)),\quad \Delta\subset\mathbb{R}_{+}.$$
Note that by Lemma \ref{measurable set}, we have  $S\chi_{\Delta}\leq 2S\chi_{(0,m(\Delta))}$ for every measurable set $\Delta\subset \mathbb{R}_{+}$ with Lebesgue measure $m.$ Hence, the latter condition can be re-stated as
$$\|S\chi_{(0,u)}\|_{\Lambda_{\varphi}(\mathbb{R}_{+})}\leq c_{\phi,\varphi}\phi(u),\quad u>0.$$
Obviously,
\begin{equation}\label{Slog}(S\chi_{(0,u)})(t)=1+\log(\frac{u}{t}),\quad t<u
\end{equation}
and
$$(S\chi_{(0,u)})(t)=\frac{u}{t},\quad t\geq u.$$
Thus,
$$\|S\chi_{(0,u)}\|_{\Lambda_{\varphi}(\mathbb{R}_{+})}=\int_0^u\big(1+\log(\frac{u}{t})\big)d\varphi(t)+\int_u^{\infty}\frac{u}{t}d\varphi(t).$$
Using integration by parts on the right-hand side, we obtain
\begin{eqnarray*}\begin{split}\|S\chi_{(0,u)}\|_{\Lambda_{\varphi}(\mathbb{R}_{+})}
&=\left(\varphi(t)\cdot\log\big(\frac{eu}{t}\big)\right)\mid_{t=0}^{t=u}\\
&+\int_0^u\frac{\varphi(t)}{t}dt+\frac{u\varphi(t)}{t}\mid_{t=u}^{t=\infty}+u\int_u^{\infty}\frac{\varphi(t)}{t^2}dt.
\end{split}\end{eqnarray*}
Taking conditions $\lim_{t\to0}\log\big(\frac{1}{t}\big)\varphi(t)=0$ and $\lim_{t\to\infty}\frac{\varphi(t)}{t}=0$ into account, we obtain
$$\|S\chi_{(0,u)}\|_{\Lambda_{\varphi}(\mathbb{R}_{+})}=\int_0^u\frac{\varphi(t)}{t}dt+u\int_u^{\infty}\frac{\varphi(t)}{t^2}dt.$$
This completes the proof.
\end{proof}

\begin{lem}\label{main nec lemma}
Let $\phi:[0,\infty)\to[0,\infty)$ be an increasing concave function such that $\phi(0+)=0$ and let $\psi$ be the function defined by the formula \eqref{psi function}.
 For every $u>0,$ there exists $y\in\Lambda_{\phi}(\mathbb{R}_{+})$ such that $\|y\|_{\Lambda_{\phi}(\mathbb{R}_{+})}\leq 2\psi(u)$ and
$$\chi_{(0,u)}\leq S\mu(y).$$
\end{lem}
\begin{proof} Choose $w>1$ such that
$$2\psi(u)\geq\frac{\phi(uw)}{1+\log(w)}.$$
By \eqref{S}, we have
$$(S\chi_{(0,uw)})(u)=1+\log(w).$$
Set $y=\mu(y)=\frac{\chi_{(0,uw)}}{1+\log(w)}.$ Obviously,
$$1= (Sy)(u)\leq (Sy)(t),\quad t<u,$$
and, therefore, $\chi_{(0,u)}\leq Sy.$

On the other hand, we have
$$\|y\|_{\Lambda_{\phi}(\mathbb{R}_{+})}=\frac{\phi(uw)}{1+\log(w)}\leq 2\psi(u).$$
This concludes the proof.
\end{proof}
The following theorem is the main result of this section.
\begin{thm}\label{main th} Let $\phi:[0,\infty)\to[0,\infty)$ be an increasing concave function such that $\phi(0+)=0$ and let $\psi$ be the function defined by the formula \eqref{psi function} and satisfying $\lim_{t\to\infty}\frac{\psi(t)}{t}=0.$ Suppose $\Lambda_{\phi}(\mathbb{R}_{+})\subset\Lambda_{\log}(\mathbb{R}_{+}).$ If
\begin{equation}\label{crit}
\int_0^u\frac{\psi(t)}{t}dt+u\int_u^{\infty}\frac{\psi(t)}{t^2}dt\leq c_{\phi,\psi}\phi(u),\quad u>0
\end{equation} holds, then
\begin{enumerate}[{\rm (i)}]
\item $S:\Lambda_{\phi}(\mathbb{R}_{+})\to\Lambda_{\psi}(\mathbb{R}_{+});$
\item for every $x\in\Lambda_{\psi}(\mathbb{R}_{+}),$ there exists $y\in\Lambda_{\phi}(\mathbb{R}_{+})$ such that $\mu(x)\leq S\mu(y)$ and $\|y\|_{\Lambda_{\phi}(\mathbb{R}_{+})}\leq8\|x\|_{\Lambda_{\psi}(\mathbb{R}_{+})}.$
\end{enumerate}
\end{thm}
\begin{proof} Let $\psi$ be the function given by the formula \eqref{psi function} such that $\lim_{t\to\infty}\frac{\psi(t)}{t}=0.$ By Lemma \ref{L2}, we have
$$\lim_{u\to0}\log(\frac1u)\psi(u)=0.$$
Moreover, by Lemma \ref{psi def lemma}, $\psi$ is an increasing concave function, and clearly $\psi(+0)=0.$
Hence, $\phi$ and $\psi$ satisfy the assumptions in Lemma \ref{L3} and we have
$$S:\Lambda_{\phi}(\mathbb{R}_{+})\to\Lambda_{\psi}(\mathbb{R}_{+}).$$
This proves the first part of the theorem.

Next, we prove the second part of the theorem. Let $x=\mu(x)\in\Lambda_{\psi}(\mathbb{R}_{+}).$ Set
$$z=\sum_{n\in\mathbb{Z}}2^n\chi_{[2^n,\infty)}(x).$$
Obviously, $z=\mu(z)$ is a step function, $z\leq 2x$ and $x\leq 2z.$
By Lemma \ref{main nec lemma}, there exists a function $y_n\in\Lambda_{\phi}(\mathbb{R}_{+})$ such that
$$\chi_{[2^n,\infty)}(x)\leq S\mu(y_n)$$ and $$\|y_n\|_{\Lambda_{\phi}(\mathbb{R}_{+})}\leq2\|\chi_{[2^n,\infty)}(x)\|_{\Lambda_{\psi}(\mathbb{R}_{+})}=2\psi\big(d_{|x|}(2^n)\big),$$
where $d_{|x|}(s):=m(\{t:|x(t)|\geq s\})$ is the distribution function of $|x|.$
We now write
$$x\leq 2z\leq\sum_{n\in\mathbb{Z}}2^nS\mu(y_n)=S\left(\sum_{n\in\mathbb{Z}}2^n\mu(y_n)\right).$$
Setting
$$y:=\sum_{n\in\mathbb{Z}}2^{n+1}\mu(y_n),$$
we obtain $x\leq S\mu(y)$ and
$$\|y\|_{\Lambda_{\phi}(\mathbb{R}_{+})}
\leq\sum_{n\in\mathbb{Z}}2^{n+1}\|y_n\|_{\Lambda_{\phi}(\mathbb{R}_{+})}\leq4\sum_{n\in\mathbb{Z}}2^{n}\psi\big(d_{|x|}(2^n)\big)=4\|z\|_{\Lambda_{\phi}(\mathbb{R}_{+})}\leq 8\|x\|_{\Lambda_{\phi}(\mathbb{R}_{+})}.
$$
\end{proof}

Suppose that $\Lambda_{\phi}(\mathbb{R}_{+})\subset \Lambda_{\log}(\mathbb{R}_{+})$. In \cite[Theorem 26]{STZ}, the authors considered a linear space
$$F(\Lambda_{\phi}):=\{x\in(L_{1,\infty}+L_{\infty})(\mathbb{R}_{+}):  \exists y\in \Lambda_{\phi}(\mathbb{R}_{+}), \, \mu(x)\leq S\mu(y)\}.$$
It was further demonstrated in \cite{STZ} that $F(\Lambda_{\phi})$ is a symmetric quasi-Banach function space when equipped with a quasi-norm
$$\|x\|_{F(\Lambda_{\phi})}:=\inf\{\|y\|_{\Lambda_{\phi}}:\mu(x)\leq S\mu(y)\}<\infty.$$
The space $(F(\Lambda_{\phi}), \|\cdot\|_{F(\Lambda_{\phi})})$ is obviously, the least receptacle (in the class of symmetric quasi-Banach  function spaces) for the operator $S$ on $\Lambda_{\phi}$.

The following result shows that the Lorentz space $\Lambda_{\psi}(\mathbb{R}_{+})$ described in the Theorem \ref{main th}  is, in fact, the least receptacle (in the class of symmetric Banach  function spaces) for the operator $S$ on
$\Lambda_{\phi}$.
\begin{corollary}\label{main th2} Let the assumptions of the Theorem \ref{main th} hold.
Then,
$$F(\Lambda_{\phi})=\Lambda_{\psi}(\mathbb{R}_{+}).$$
\end{corollary}
\begin{proof}
By Theorem \ref{main th} (i), we have
$$S:\Lambda_{\phi}(\mathbb{R}_{+})\to\Lambda_{\psi}(\mathbb{R}_{+}).$$
Hence, it follows from \cite[Theorem 26 (ii)]{STZ} that
$F(\Lambda_{\phi})\subset\Lambda_{\psi}(\mathbb{R}_{+}).$

To prove converse inclusion, if $x\in \Lambda_{\psi}(\mathbb{R}_{+}),$ then by Theorem \ref{main th} (ii) there exist $y\in\Lambda_{\phi}(\mathbb{R}_{+})$ such that
$\mu(x)\leq S\mu(y).$ This shows $\Lambda_{\psi}(\mathbb{R}_{+})\subset F(\Lambda_{\phi}).$
\end{proof}

\begin{corollary}\label{opt. range th H} Let the assumptions of the Theorem \ref{main th} hold.
Then the space $\Lambda_{\psi}(\mathbb{R})$ is the least symmetric Banach function space, that is the optimal symmetric range for the Hilbert transformation $\mathcal{H}$ defined in \eqref{hilbert tr}. The mapping
$$\mathcal{H}:\Lambda_{\phi}(\mathbb{R})\rightarrow\Lambda_{\psi}(\mathbb{R})$$
is bounded.
\end{corollary}
\begin{proof}
By Theorem \ref{main th} (i), we have $S:\Lambda_{\phi}(\mathbb{R}_{+})\to\Lambda_{\psi}(\mathbb{R}_{+}).$ Hence, by \cite[Theorem III.4.8, p. 138]{BSh}, $\mathcal{H}:\Lambda_{\phi}(\mathbb{R})\rightarrow\Lambda_{\psi}(\mathbb{R})$ is bounded.
Now, suppose that $G(\mathbb{R})$ is another symmetric Banach function space such that $\mathcal{H}:\Lambda_{\psi}(\mathbb{R})\rightarrow G(\mathbb{R})$ is bounded.
 Take $x\in \Lambda_{\phi}(\mathbb{R}).$ By Lemma \ref{dom of Hilbert trans} there exists $y$ with $\mu(x)=\mu(y)$ such that $S\mu(x)\leq c_{abs}\mu(\mathcal{H}y),$ which shows that $S\mu(x)\in G(\mathbb{R}_{+}).$
 Since $x\in \Lambda_{\phi}(\mathbb{R})$ is arbitrary, it follows from
 the Corollary \ref{main th2} that $\Lambda_{\psi}(\mathbb{R}_{+})\subset G(\mathbb{R}_{+}).$

This completes the proof.
\end{proof}

\section{Optimal range for the triangular truncation operator in Schatten-Lorentz ideals}\label{T}

In this section, we investigate optimal range of the triangular truncation operator $T$ in Schatten-Lorentz ideals.
Let $\phi:[0,\infty)\to[0,\infty)$ be an increasing concave function such that $\phi(0+)=0$
and let $\psi$ be the function defined by the formula \eqref{psi function}.
Let $\Lambda_{\phi}(\mathbb{Z}_{+})$ be the Lorentz sequence space associated with a function $\phi.$
If $E$ is a symmetric sequence space on $\mathbb{Z}_{+}$ such that $E(\mathbb{Z}_{+})\subset\Lambda_{\log}(\mathbb{Z}_{+}),$ then the operator
 $$S^{d}:E(\mathbb{Z}_{+})\rightarrow \ell_{\infty}(\mathbb{Z}_{+})$$
 is well defined (see \eqref{S dis}).


The optimal symmetric range for the discrete Calder\'{o}n operator $S^{d}$ on Lorentz sequence space $\Lambda_{\phi}(\mathbb{Z}_{+})$ can be characterized similarly to the case of Lorentz function spaces in Corollary \ref{main th2}.
Indeed, setting
$$F(\Lambda_{\phi})(\mathbb{Z}_{+}):=\{a\in\ell_{\infty}(\mathbb{Z}_{+}):  \exists b\in \Lambda_{\phi}(\mathbb{Z}_{+}), \, \mu(a)\leq S^{d}\mu(b)\}$$
and
$$\|a\|_{F(\Lambda_{\phi})(\mathbb{Z}_{+})}:=\inf\{\|b\|_{\Lambda_{\phi}(\mathbb{Z}_{+})}:\mu(a)\leq S^{d}\mu(b)\}<\infty$$
and repeating the arguments in  \cite[Theorem 26 (i)]{STZ} and in Corollary \ref{main th2} above, we conclude that $(F(\Lambda_{\phi})(\mathbb{Z}_{+}), \|\cdot\|_{F(\Lambda_{\phi}})$ is a symmetric quasi-Banach sequence space such that 
$$F(\Lambda_{\phi})(\mathbb{Z}_{+})=\Lambda_{\psi}(\mathbb{Z}_{+}),$$
where $\psi$ is given by formula \eqref{psi function}.


The following theorem is a non-commutative analogue of the Theorem \ref{main th} for the triangular truncation operator $T$ defined in \eqref{T-oper}.
\begin{thm}\label{opt. range th T} Let the assumptions of Theorem \ref{main th} hold. The space $\Lambda_{\psi}(H)$ is the least ideal in $B(H)$ such that
$$T:\Lambda_{\phi}(H)\rightarrow\Lambda_{\psi}(H).$$
\end{thm}
\begin{proof}
First, let us see that $T:\Lambda_{\phi}(H)\rightarrow\Lambda_{\psi}(H)$ is bounded. By \cite[Theorem 11]{STZ}, $T$ satisfies the assumptions of \cite[Theorem 14 (ii)]{STZ}. Therefore, we have $\mu(T(A))\lesssim S\mu(A).$ By Lemmas \ref{psi def lemma} and \ref{L2} functions $\phi$ and $\psi$ satisfy the conditions of Lemma \ref{L3}. By the latter lemma, we have $S:\Lambda_{\phi}(0,\infty)\to\Lambda_{\psi}(0,\infty),$ and so
\begin{eqnarray*}\begin{split}
\|T(A)\|_{\Lambda_{\psi}(H)}
&=\|\mu(T(A))\|_{\Lambda_{\psi}(\mathbb{Z}_{+})}\lesssim\|S\mu(A)\|_{\Lambda_{\psi}(0,\infty)}\\
&\leq c_{abs}\|\mu(A)\|_{\Lambda_{\phi}(0,\infty)}=c_{abs}\|A\|_{\Lambda_{\phi}(H)}, \quad \forall A\in \Lambda_{\phi}(H).
\end{split}\end{eqnarray*}

Now, suppose that $\mathcal{G}(H)$ is another ideal in $B(H)$ such that $T:\Lambda_{\phi}(H)\rightarrow \mathcal{G}(H).$ Let us show that $\Lambda_{\psi}(H)\subset \mathcal{G}(H).$
Let $a\in \Lambda_{\phi}(\mathbb{Z}_{+}).$ By \cite[Theorem 21]{STZ}, there exists an operator $A$ such that $\mu(a)=\mu(A)$ and $S^{d}\mu(a)\lesssim\mu(T(A)).$ Since $T(A)\in \mathcal{G}(H),$ it follows that $S^{d}\mu(a)\in G(\mathbb{Z}_{+}),$ where sequence space $G(\mathbb{Z}_+)$ is obtained from the ideal $\mathcal{G}(H)$ by the Calkin correspondence \cite{LSZ}. Thus, $S^d:\Lambda_{\phi}(\mathbb{Z}_+)\to G(\mathbb{Z}_+).$

By Theorem \ref{main th}, for every $b\in\Lambda_{\psi}(\mathbb{Z}_+),$ there exists $a\in\Lambda_{\phi}(0,\infty)$ such that $\mu(b)\leq S\mu(a).$ For every $x\in (L_1+L_{\infty})(0,\infty)$ straightforward computation yields
$$S\mu(x)\leq 4S^d\mu(z)\mbox{ on }(1,\infty),\mbox{ where }z=\Big\{\int_n^{n+1}\mu(s,x)\Big\}_{n\geq0}.$$
Taking $x=a$ and setting $c=4z\in\Lambda_{\phi}(\mathbb{Z}_+),$ we infer that
$$\mu(b)\leq (\|b\|_{\infty},0,0,\cdots)+S_d\mu(c).$$
Hence, $b\in G(\mathbb{Z}_+).$ Since $b\in\Lambda_{\psi}(\mathbb{Z}_+)$ is arbitrary, it follows that $\Lambda_{\psi}(H)\subset \mathcal{G}(H),$ thereby completing the proof.
\end{proof}

We now illustrate our methods by comparing our results with that of Arazy \cite{Ar} (in the special case of Schatten-Lorentz ideals).
\begin{thm}\label{Tphitophi} Let $\phi$ be an increasing concave function on $[0,\infty)$ such that $\phi(0+)=0.$ Suppose $\Lambda_{\phi}(\mathbb{Z}_{+})\subset\Lambda_{\log}(\mathbb{Z}_{+}).$
The following are equivalent:
\begin{enumerate}[{\rm (i)}]
\item $T:\Lambda_{\phi}(H)\to \Lambda_{\phi}(H);$
\item $S^d:\Lambda_{\phi}(\mathbb{Z}_{+})\to \Lambda_{\phi}(\mathbb{Z}_{+});$
\item there exists a constant $c_{\phi}>0$ such that
$$\frac{1}{n+1}\sum_{k=1}^{n}\frac{\phi(k)}{k}+\sum_{k=n+1}^{\infty}\frac{\phi(k)}{k^2}\leq c_{\phi}\frac{\phi(n)}{n},\quad n\geq1.$$
\end{enumerate}
\end{thm}
\begin{proof} By Theorem 14 (ii) and Theorem 21 in \cite{STZ} $T:\Lambda_{\phi}(H)\to \Lambda_{\phi}(H)$ if and only if $S^d:\Lambda_{\phi}(\mathbb{Z}_{+})\to \Lambda_{\phi}(\mathbb{Z}_{+}).$ This shows equivalence of $(i)$ and $(ii).$ Following the argument in Lemma \ref{L3} {\it mutatis mutandi}, we obtain the equivalence of $(ii)$ and $(iii).$
\end{proof}

\section{Lipschitz and commutator estimates in Schatten-Lorentz ideals}\label{LC}
In this section, we present an application of our approach in previous sections to Double Operator Integrals. In particular,
we obtain Lipschitz and commutator estimates in  Schatten-Lorentz ideals of compact operators described in the previous section.

Let $A$ be a bounded self-adjoint operator on $H$ and $\xi$ be a bounded Borel function on $\mathbb{R}^2$. Symbolically, a double operator integral is defined by the formula
\begin{equation}\label{doi def}
T_{\xi}^{A,A}(V)=\int_{\mathbb{R}^2}\xi(\lambda,\mu)dE_A(\lambda)VE_A(\mu),\quad V\in \mathcal{L}_2(H).
\end{equation}
For a more rigorous definition, consider projection valued measures on $\mathbb{R}$ acting on the Hilbert space $\mathcal{L}_2(H)$ by the formula $X\to E_A(\mathcal{B})X$ and $X\to XE_A(\mathcal{B}).$ These spectral measures commute and, hence (see Theorem
V.2.6 in \cite{BirSol}), there exists a countably additive (in the strong operator topology) projection-valued measure $\nu$ on $\mathbb{R}^2$ acting on the Hilbert space $\mathcal{L}_2(H)$ by the formula
$$\nu(\mathcal{B}_1\otimes\mathcal{B}_2):X\to E_A(\mathcal{B}_1)XE_A(\mathcal{B}_2),\quad X\in \mathcal{L}_2(H).$$
Integrating a bounded Borel function $\xi$ on $\mathbb{R}^2$ with respect to the measure $\nu$ produces a bounded operator acting on the Hilbert space $\mathcal{L}_2(H).$ In what follows, we denote the latter operator by
$T_{\xi}^{A,A}$ (see also \cite[Remark 3.1]{PSW}).

We are mostly interested in the case $\xi=f^{[1]}$ for a Lipschitz function $f:\mathbb{R}\rightarrow \mathbb{C}.$ Here,
\begin{equation}\label{lip}
f^{[1]}(\lambda,\mu)=
\begin{cases}
\frac{f(\lambda)-f(\mu)}{\lambda-\mu},\quad \lambda\neq\mu\\
0,\quad \lambda=\mu.
\end{cases}
\end{equation}

The following is the main result of this section which describes Lipschitz and commutator estimates in Schatten-Lorentz ideals of compact operators.
\begin{thm}\label{DOIth} Let the assumptions of Theorem \ref{main th} hold.
The following assertions hold:
\begin{enumerate}[{\rm (i)}]
\item If $A=A^{*}$ is a self-adjoint operator in $B(H)$ and  $f:\mathbb{R}\rightarrow \mathbb{R}$ is a Lipschitz function, then the double operator integral (associated with the function $f^{[1]}$ defined in \eqref{lip}) $T_{f^{[1]}}^{A,A}:\Lambda_{\phi}(H)\rightarrow \Lambda_{\psi}(H)$ is bounded and
$$\|T_{f^{[1]}}^{A,A}\|_{\Lambda_{\phi}(H)\rightarrow \Lambda_{\psi}(H)}\leq c_{\mathbf{\Lambda_{\phi}}}\|f'\|_{L_{\infty}(\mathbb{R})};$$
\item  For all self-adjoint operators $A,B\in B(H)$ such that $[A,B]\in \Lambda_{\psi}(H)$ and for every Lipschitz function $f:\mathbb{R}\rightarrow \mathbb{C},$ we have
$$\|[f(A),B]\|_{\Lambda_{\psi}(H)}\lesssim\|f'\|_{L_{\infty}(\mathbb{R})}\|[A,B]\|_{\Lambda_{\phi}(H)},$$
where $[A,B]:=AB-BA.$
For all self-adjoint operators $X,Y\in B(H)$ such that $X-Y\in\Lambda_{\phi}(H)$ and for every Lipschitz function  $f:\mathbb{R}\rightarrow\mathbb{R},$ we have
$$\|f(X)-f(Y)\|_{\Lambda_{\psi}(H)}\lesssim\|f'\|_{L_{\infty}(\mathbb{R})}\|X-Y\|_{\Lambda_{\phi}(H)}.$$
\end{enumerate}
\end{thm}

\begin{proof} $(i).$ By Theorem 1.2 in \cite{CPSZ}, we have
$$\|T_{f^{[1]}}^{A,A}(V)\|_{\mathcal{L}_{1,\infty}(H)}\lesssim \|f'\|_{L_{\infty}(\mathbb{R})}\|V\|_{\mathcal{L}_{1}(H)}, \quad \forall V\in (\mathcal{L}_{1}\cap \mathcal{L}_{2})(H).$$
Thus, the operator $T_{f^{[1]}}^{A,A}$ satisfies the conditions of Theorem 14 (ii) in \cite{STZ}. In other words, we have
$$\mu(T_{f^{[1]}}^{A,A}(V))\lesssim\|f'\|_{L_{\infty}(\mathbb{R})} S\mu(V), \quad V\in \Lambda_{\log}(H).$$
By Theorem \ref{main th} (i), the operator $S$ acts boundedly from $\Lambda_{\phi}(0,\infty)$ into $\Lambda_{\psi}(0,\infty).$
We have
\begin{eqnarray*}\begin{split}
\|T_{f^{[1]}}^{A,A}(V)\|_{\Lambda_{\psi}(H)}
&=\|\mu\big(T_{f^{[1]}}^{A,A}(V)\big)\|_{\Lambda_{\psi}(\mathbb{Z}_{+})}\lesssim\|f'\|_{L_{\infty}(\mathbb{R})}\|S\mu(V)\|_{\Lambda_{\psi}(0,\infty)}\\
&\lesssim\|f'\|_{L_{\infty}(\mathbb{R})}\|\mu(V)\|_{\Lambda_{\phi}(0,\infty)}=c_{abs}\|f'\|_{L_{\infty}(\mathbb{R})}\|V\|_{\Lambda_{\phi}(H)}, \,\ V\in \Lambda_{\phi}(H).
\end{split}\end{eqnarray*}
In other words, $T_{f^{[1]}}^{A,A}:\Lambda_{\phi}(H)\rightarrow \Lambda_{\psi}(H)$ is bounded.

$(ii).$
The double operator integral $T^{A,A}_{f^{[1]}}([A,B])$ is equal to $[f(A),B]$  for the operators $A,B\in B(H)$ such that $[A,B]\in \Lambda_{\phi}(H)$ (see \cite[Proposition 2.6]{PS}). Therefore, the commutator estimate follows from part $(i).$ Finally, applying the commutator estimate to the operators
$$A=\left(
\begin{array}{cc}
X & 0 \\
0 & Y \\
\end{array}
\right), \quad B=\left(
\begin{array}{cc}
0 & 1 \\
1 & 0 \\
\end{array}
\right),
$$
we obtain Lipschitz estimate.

\end{proof}
In particular, we obtain the following result.
\begin{thm} Let the assumptions of Theorem \ref{Tphitophi} hold. If any one of the following holds
\begin{enumerate}[{\rm (i)}]
\item $T:\Lambda_{\phi}(H)\to \Lambda_{\phi}(H);$
\item $S^d:\Lambda_{\phi}(\mathbb{Z}_{+})\to \Lambda_{\phi}(\mathbb{Z}_{+});$
\item there exists a constant $c_{\phi}>0$ such that
$$\frac{1}{n+1}\sum_{k=1}^{n}\frac{\phi(k)}{k}+\sum_{k=n+1}^{\infty}\frac{\phi(k)}{k^2}\leq c_{\phi}\frac{\phi(n)}{n},\quad n\geq1.$$
\end{enumerate} then
\begin{enumerate}[{\rm (i)}]
\item The double operator integral $T_{f^{[1]}}^{A,A}:\Lambda_{\phi}(H)\rightarrow \Lambda_{\phi}(H)$ is bounded and
$$\|T_{f^{[1]}}^{A,A}\|_{\Lambda_{\phi}(H)\rightarrow \Lambda_{\phi}(H)}\leq c_{\mathbf{\Lambda_{\phi}}}\|f'\|_{L_{\infty}(\mathbb{R})}$$
for all self-adjoint operator $A=A^{*}$ in $B(H)$ and Lipschitz function $f:\mathbb{R}\rightarrow \mathbb{R};$

\item  For all self-adjoint operators $A,B\in B(H)$ such that $[A,B]\in \Lambda_{\phi}(H)$ and for every Lipschitz function $f:\mathbb{R}\rightarrow \mathbb{C},$ we have
$$\|[f(A),B]\|_{\Lambda_{\phi}(H)}\lesssim\|f'\|_{L_{\infty}(\mathbb{R})}\|[A,B]\|_{\Lambda_{\phi}(H)},$$
where $[A,B]:=AB-BA.$
For all self-adjoint operators $X,Y\in B(H)$ such that $X-Y\in\Lambda_{\phi}(H)$ and for every Lipschitz function  $f:\mathbb{R}\rightarrow\mathbb{R},$ we have
$$\|f(X)-f(Y)\|_{\Lambda_{\phi}(H)}\lesssim\|f'\|_{L_{\infty}(\mathbb{R})}\|X-Y\|_{\Lambda_{\phi}(H)}.$$

\end{enumerate}
\end{thm}
\begin{proof} The argument follows the same line as the proof of Theorems \ref{DOIth} and \ref{opt. range th T} and is therefore omitted.
\end{proof}

\section{Acknowledgement}

The first and third authors were partially supported by Australian Research Council. The second author was partially supported by the grants AP08052004 and AP08051978 of the Science Committee of the Ministry of Education and Science of the Republic of Kazakhstan.

\end{document}